\documentclass[preprint,12pt]{elsarticle}

\usepackage{verbatim}
\usepackage{cite}
\usepackage{enumerate}
\usepackage{graphicx}
\usepackage{mathtools}
\usepackage{amsmath}
\usepackage{amsthm}
\usepackage{amsfonts}
\usepackage{amssymb}
\usepackage{hyperref}
\usepackage[dvipsnames]{xcolor}
\usepackage{caption}
\usepackage{subcaption}

\newtheorem{theorem}{Theorem}[section]
\newtheorem{proposition}[theorem]{Proposition}
\newtheorem{corollary}[theorem]{Corollary}
\newtheorem{lemma}[theorem]{Lemma}
\newtheorem{definition}[theorem]{Definition}
\newtheorem{example}[theorem]{Example}
\newtheorem{remark}[theorem]{Remark}

\newcommand{\R}{{\mathbb{R}}}

\newcommand{\N}{{\mathbb{N}}}

\DeclareMathOperator{\pr}{pr}

\begin{document}

\begin{frontmatter}

\title{Rigorous Enclosures of Solutions of Neumann Boundary Value Problems}

\author[icmc]{Eduardo Ramos}
\ead{eduardoramos@usp.br}

\author[icmc]{Victor Nolasco\corref{cor1}}
\ead{victor.nolasco@usp.br}

\author[icmc,rutgers]{Marcio Gameiro}
\ead{gameiro@icmc.usp.br}

\cortext[cor1]{Corresponding author}

\address[icmc]{Instituto de Ci\^{e}ncias Matem\'{a}ticas e de Computa\c{c}\~{a}o, Universidade de S\~{a}o Paulo, 13560-970, S\~{a}o Carlos, SP, Brazil}

\address[rutgers]{Department of Mathematics, Rutgers University, Piscataway, NJ}

\begin{abstract}
This paper is dedicated to the problem of isolating and validating zeros of non-linear two point boundary value problems. We present a method for such purpose based on the Newton-Kantorovich Theorem to rigorously enclose isolated zeros of two point boundary value problem with Neumann boundary conditions.
\end{abstract}


\end{frontmatter}

\section{Introduction}

Rigorous numerical methods for the verification of existence of solutions of non-linear differential equations are of great importance due to the fact that, although it is usually possible to solve such problems numerically, it is often very hard, if not impossible, to obtain analytical solutions. Additionally such non-linear problems frequently have multiple solutions, whose multiplicity and behaviour often depend on parameters. Hence rigorous numerical methods to prove the existence, multiplicity and provide additional information about the behaviour of solutions to non-linear differential equations are fundamental tools for non-linear analysis and applications. Many authors have proposed verification methods to solve partial differential equations such as the analytical method based on a Newton-Kantorovich theorem in \citep{1995-Plum}, the semi-analytical method in \citep{1988-Nakao, 1995-Nakao}, and the analytical method based on radii polynomials in \citep{2010-Gameiro, 2016-Gameiro}. The methods in \citep{2010-Gameiro, 2016-Gameiro} are based on a contraction mapping argument applied to a Newton-Like operator $T$ on a Banach space, in which the search for zeros of a map $\mathcal{F}$ is shown to be equivalent to finding fixed points of $T$. This method is useful to obtain equilibria \citep{gameiro2008validated, gameiro2011rigorous}, invariant manifolds \citep{castelli2015parameterization, castelli2018parameterization}, connecting orbits \citep{van2011rigorous, lessard2014computer}, solutions in spaces with varying regularity \citep{lessard2017varying}, and other types of solutions. However, these methods are usually computationally expensive due to the fact that they often require the computation of approximate inverses (or estimates on the norm of the inverse) of large matrices. Methods to solve boundary values problems are presented in \citep{Sing.Sub.ZhangBVP1995, CHERPION200175, VERMA20114709} where upper and lower solutions are computed to approximate the solution to the boundary value problem. The approach presented in this paper differs from the methods above in the sense that it is a method to validate a previously computed numerical solution and not a method to compute the approximate solution itself. Hence this paper can be seen as a complement to other methods to compute numerical solutions to boundary value problems.

The goal of this paper is to propose an efficient method to solve rigorously a general two points boundary value problem with Neumann boundary condition of the form
\begin{equation}
\label{equation_roots}
u'' = f(x,u,u'), \quad u'(0)=u'(s)=0.
\end{equation}

To describe the main ideas of the method let us interpret the problem
as a problem of the form $\mathcal{F}(u) = 0$ for $u$ in an appropriate Hilbert space, where $\mathcal{F}(u) = u'' - f(x,u,u')$. The first step is to obtain a non-rigorous numerical candidate $w$ for a zero of $\mathcal{F}$ using a finite dimensional approximation of $\mathcal{F}$. We then use a version of the
Newton-Kantorovich theorem to rigorously verify the existence of a true zero of $\mathcal{F}$ close to $w$, based on the computation of rigorous enclosures for
\begin{equation*}
\|\mathcal{F}(w) \|, \quad \lambda(D\mathcal{F}(w)),
\end{equation*}
and a Lipschitz constant $K>0$ for $D\mathcal{F}$, where $\lambda$ is the bijectivity modulus to be defined later. The computation of these bounds are done via efficient methods based on rigorous integration, a finite dimensional approximation of $D\mathcal{F}(w)$, and a certain bound $K$ on the second partial derivatives of $f$.

This paper is organized as follows: In Section~\ref{sec:Preliminaries} we introduce some definitions and preliminaries results to be used in the paper. In Section~\ref{sec:Isolation_zeros} we present a reformulation of the Newton-Kantorovich theorem based on the bijectivity modulus and show how to compute the constants needed in this reformulation. Section~\ref{sec:Applications} is dedicated to apply the general theoretical results of Section~\ref{sec:Isolation_zeros} to rigorously compute solutions to \eqref{equation_roots}. Finally, in Section~\ref{sec:Conclusion} we present our conclusions and plans for future work.

\section{Preliminaries}
\label{sec:Preliminaries}

In this section we present some definitions and preliminary results needed to prove the main theorem of this paper (Theorem~\ref{kantorovich2}), which uses a modified version of the Newton-Kantorovich theorem to prove existence of solutions to boundary values problems.

Throughout this paper we will denote $I = (0,1)$ and $\bar{I} = [0,1]$. Given $u \in L^2(I)$, we define the cosine and sine series expansions of $u$ by
\[
u(x) \sim \widehat{u}_{\cos}(1) + \sum_{k=2}^\infty \widehat{u}_{\cos}(k) \sqrt{2}\cos((k-1)\pi x)
\]
and
\[
u(x) \sim \sum_{k=2}^\infty \widehat{u}_{\sin}(k) \sqrt{2}\sin((k-1)\pi x)
\]
respectively, where
\begin{equation*}
\widehat{u}_{\cos}(1)=\int_0^1 u(x)dx, \quad
\widehat{u}_{\cos}(k) = \int_0^1 u(x)\sqrt{2} \cos((k-1)\pi x) dx \quad \text{for}~ k\geq 2
\end{equation*}
and
\begin{equation*}
\widehat{u}_{\sin}(k)= \int_0^1 u(x)\sqrt{2} \sin((k-1)\pi x) dx \quad \text{for}~ k\geq 2.
\end{equation*}
Moreover (see \citep[p. 145]{2011-Brezis}) the sets
\begin{equation*}
\left\{1, \sqrt{2} \cos(\pi x),\sqrt{2} \cos(2\pi x),\cdots\right\} \quad \text{and} \quad \left\{\sqrt{2} \sin(\pi x),\sqrt{2} \sin(2\pi x),\cdots\right\}
\end{equation*}
are orthonormal bases for $L^{2}(I)$. The following is a well know consequence of the Parseval formula (see \citep[Theorem 7.6]{1980-Taylor}).

\begin{proposition}
\label{isometry}
Let $X$ be a Hilbert space with orthonormal basis $B_X=\{ e_1,e_2,\cdots \}$. Given $x\in X$ let $\pi_{X}(x) = \left(\langle x,e_k \rangle\right)_{k\in \N}$. Then $\pi_X(x) \in \ell^2(\N)$ for all $x\in X$ and $\pi_X \colon L^2(I)\to \ell^2(\N)$ is an isometric isomorphism.
\end{proposition}

Thus, denoting $\pi_{\cos}(u)=(\widehat{u}_{\cos}(1),\widehat{u}_{\cos}(2),\cdots)$ and $\pi_{\sin}(u)=(\widehat{u}_{\sin}(2),\widehat{u}_{\sin}(3),\cdots)$, for $u\in L^{2}(I)$, the following corollary follows directly from the proposition above.

\begin{corollary}
\label{parsevalcos}
Given $u\in L^2(I)$ we have that $\pi_{\cos}(u), \pi_{\sin}(u) \in \ell^2(\N)$ and moreover $\pi_{\cos} \colon L^2(I) \to \ell^2(\N)$ and $\pi_{\sin} \colon L^2(I)\to \ell^2(\N)$ are isometric isomorphisms.
\end{corollary}
We denote by $H^q(I)$ the Sobolev space of functions $u \in L^{2}(I)$ whose $m$-th weak derivative $u^{(m)}$ exists and is square integrable for all $0 \leq m\leq q$. The space $H^q(I)$ is Hilbert with the usual inner product
\begin{equation*}
\left\langle u,v\right\rangle_{H^q(I)}= \sum_{k=0}^q \left\langle u^{(k)},v^{(k)}\right\rangle_{L^2(I)}.
\end{equation*}
It is well known (see \citep[Theorem 8.2]{2011-Brezis}) that every function $u \in H^q(I)$ has a unique representative $\bar{u} \in C^{q-1}(\bar{I})$, that is, such that $u=\bar{u}$ a.e. on $I$. Thus we assume, without loss of generality, that $u(x)=\bar{u}(x)$ for all $x \in \bar{I}$. Our method looks for zeros of $\mathcal{F}$ in the space $H^2_N(I)$ defined below.

\begin{definition}
We denote by $H^2_N(I)$ the subspace of $H^2(I)$ consisting of the functions $u\in H^2(I)$ such that $u'(0)=u'(1)=0$. Furthermore, $H^1_0(I)$ denotes the subspace of functions $u\in H^1(I)$ satisfying $u(0)=u(1)=0$.
\end{definition}

The following proposition follows directly by integration by parts for functions in $H^1(I)$.

\begin{proposition}
\label{lemma0}
If $u\in H^1(I)$ then
\begin{enumerate}[$(i)$]
\item $\widehat{u'}_{\sin}(k) = -(k-1)\pi\widehat{u}_{\cos}(k)$ for all $k\in \N$, $k>2$;
\item $\widehat{u'}_{\cos}(k)=(k-1)\pi \widehat{u}_{\sin}(k)$ for all $k\in \N$ if $u\in H^1_0(I)$.
\end{enumerate} 
\end{proposition}

\begin{proposition}
\label{hsinhcos}
The set
\begin{equation*}
B_{H^2_N(I)} = \left\{1,\frac{\sqrt{2}\cos(\pi x)}{\omega(2)},\frac{\sqrt{2}\cos(2 \pi x)}{\omega(3)},\cdots \right\},
\end{equation*}
where $\omega(k) := \sqrt{1+((k-1) \pi)^2+((k-1) \pi)^4}$ for $k\in \N$, is an orthonormal basis for $H^2_N(I)$.
\end{proposition}

\begin{proof}
We can directly see that $B_{H^2_N(I)}$ forms an orthonormal set in $H^2_N(I)$. Now, to prove that the linear span of $B_{H^2_N(I)}$ is dense in $H^2_N(I)$, given $u\in H^2_N(I)$, consider $(u_m)_{m\geq 2}$ defined by 

\begin{equation*}
u_m(x) = \sum_{k=2}^{m}\widehat{u}_{\sin}(k)\sqrt{2}\sin((k-1)\pi x).
\end{equation*}
It is clear $u_m$ is in the linear span of $B_{H^2_N(I)}$. Now, let $w_m=u-u_m$ for all $m\in \N$, $m\geq 2$. Since $\widehat{(w_m)}_{\cos}(k)=0$ for all $1\leq k\leq m$ and $\widehat{(w_m)}_{\cos}(k)=\widehat{u}_{\cos,m}(k)$ for $k\geq m+1$, by Propositions~\ref{parsevalcos} and \ref{lemma0} we have
 
\begin{equation*}
 \|u-u_m \|_{H^2(I)} = \sqrt{\sum_{k=m+1}^\infty \widehat{u}_{\cos}(k)^2 + \sum_{k=m+1}^\infty \widehat{u'}_{\sin}(k)^2 + \sum_{k=m+1}^\infty \widehat{u''}_{\cos}(k)^2},
\end{equation*}
but since $u\in H^2_N(I)$, it follows by Proposition~\ref{parsevalcos} that $\widehat{u}_{\cos}\in \ell^2(\N)$, $\widehat{u'}_{\sin}\in \ell^2(\N)$, and $\widehat{u''}_{\cos}\in \ell^2(\N)$ and thus by the above equation we have $\lim_{m\to \infty} \|u-u_m \|_{H^2(I)}=0$, proving thus the density of the linear span of $B_{H^2_N(I)}$ over $H^2_N(I)$.
\end{proof}

Thus, denoting $h_{\cos}(u) = \left(\omega(1)\widehat{u}_{\cos}(1),\omega(2)\widehat{u}_{\cos}(2),\omega(3)\widehat{u}_{\cos}(3),\cdots\right)$ for all $u\in H^2_N(I)$, as a direct  consequence of Proposition~\ref{isometry} and Proposition~\ref{hsinhcos} we have the following.

\begin{corollary}\label{hc_iso_iso}
For all $u\in H^2_N(I)$ we have that $h_{\cos}(u)\in \ell^2(\N)$ and $h_{\cos} \colon H^2_N(I)\to \ell^2(\N)$ is an isometric isomorphism.
\end{corollary}
\begin{proof}
Let us prove that $h_{\cos} \colon H^2_N(I)\to \ell^2(\N)$ is an isometric isomorphism. Indeed, $h_{\cos}$ is linear since its entries are the coefficients of cosine basis of $L^{2}(I)$ multiplied by the weights $\omega(k) = \sqrt{1+((k-1) \pi)^2+((k-1) \pi)^4}$ for $k\in \N$. By Proposition \ref{hsinhcos} the set
\begin{equation*}
    B_{H^2_N(I)} = \left\{1,\frac{\sqrt{2}\cos(\pi x)}{\omega(2)},\frac{\sqrt{2}\cos(2 \pi x)}{\omega(3)},\cdots \right\},
\end{equation*}
is an an orthonormal basis for $H^2_N(I)$. Notice that $h_{\cos} \colon H^2_N(I)\to \ell^2(\N)$, sends the $m$-th element of $B_{H^2_N(I)}$ to the $m$-th element of the canonical basis of $\ell^2(\N)$. Thus, the result follows by Proposition~\ref{isometry} and Proposition~\ref{lemma0} by computing the coefficients of $u \in H^2_N(I)$ in the $B_{H^2_N(I)}$ basis, that is, 
denoting $c_k (t)= \sqrt{2}\cos((k-1)\pi t)$ and $s_k(t)= \sqrt{2}\sin((k-1)\pi t)$, for $k\in \N$, we get that
\begin{equation*}
\langle u, 1\rangle_{H^2(I)} = \int_0^1 u(t) dt
= \widehat{u}_{\cos}(1) = (h_{\cos} (u))(1)
\end{equation*}
and
\begin{equation*}
\begin{aligned}
& \left\langle u, \frac{\sqrt{2}\cos((k-1) \pi t)}{\omega(k)}\right\rangle_{H^2(I)} = \frac{1}{\omega(k)} \int_0^1 u(t) \sqrt{2}\cos((k-1) \pi t) dt \\
& - \frac{(k-1) \pi}{\omega(k)} \int_0^1 u'(t)\sqrt{2}\sin((k-1) \pi t) dt
- \frac{((k-1) \pi)^{2}}{\omega(k)}\int_0^1 u''(t) \sqrt{2}\cos((k-1) \pi t) dt \\
& = \frac{1}{\omega(k)}\left(\int_0^1 u(t) c_k(t) dt - ((k-1)\pi)\int_0^1 u'(t) s_k(t) dt - ((k-1)\pi)^2 \int_0^1 u''(t) c_k(t) dt\right) \\
& = \frac{1}{\omega(k)}\left(\widehat{u}_{\cos}(k) - ((k-1)\pi)\widehat{u'}_{\sin}(k) - ((k-1)\pi)^2 \widehat{u''}_{\cos}(k)\right) \\
& = \frac{1}{\omega(k)}\left(\widehat{u}_{\cos}(k) + ((k-1)\pi)^2 \widehat{u}_{\cos}(k) - ((k-1)\pi)^3 \widehat{u'}_{\sin}(k)\right) \\
& = \frac{1}{\omega(k)}\left(\widehat{u}_{\cos}(k) + ((k-1)\pi)^2 \widehat{u}_{\cos}(k) + ((k-1)\pi)^4 \widehat{u}_{\cos}(k)\right) \\
& = \frac{\omega(k)^2}{\omega(k)} \widehat{u}_{\cos}(k) = \omega(k)\; \widehat{u}_{\cos}(k) = (h_{\cos} (u))(k),
\end{aligned}
\end{equation*}
for all $k>1$.
\end{proof}

In this paper we identify $\R^m$ with a subset of $\ell^2(\N)$ through the isometric embedding $\pi_{\R^m}(a)=(a_1,\cdots,a_m, 0, \cdots)\in \ell^2(\N)$, for $a = (a_1,\cdots,a_m) \in \R^m $. In the following, we let $\pi_{\cos,m} \colon L^2(I)\to \R^m$ be given by $\pi_{\cos,m}(u) =\left (\hat{u}_{\cos}(1),\cdots,\hat{u}_{\cos}(m)\right)$ and let $h^{-1}_{\cos,m} \colon \R^m\to H^2_N(I)$ given by $h^{-1}_{\cos,m}=(h_{\cos})^{-1} \mid_{\R^m}$ be the restriction of $(h_{\cos})^{-1}$ to $\R^m$ via the above identification.

\begin{definition}
\label{F_cosm}
Given $\mathcal{F} \colon H^2_N(I)\to L^2(I)$ and $m\geq 1$ we define  $\mathcal{F}_{\cos,m} \colon \R^m\to \R^m$ by
\begin{equation*}
\mathcal{F}_{\cos,m}=\pi_{\cos,m} \circ \mathcal{F}\circ h_{\cos,m}^{-1}.
\end{equation*}
\end{definition}
As a consequence of the above definition, if $\mathcal{F} \colon H^2_N(I)\to L^2(I)$ is differentiable at $w\in H^2_N(I)$ then $D\mathcal{F}(w)_{\cos,m} \colon \R^m\to \R^m$ is defined by $D\mathcal{F}(w)_{\cos,m} = \pi_{\cos,m}\circ  D\mathcal{F}(w)\circ h^{-1}_{\cos,m}$. The function $D\mathcal{F}(w)_{\cos,m}$ plays an important role in the next section, as a suitable finite dimensional approximation to $D\mathcal{F}(w)$.

Regarding the differentiablity of $\mathcal{F}$ representing two-point boundary value problems without boundary conditions we have the following result.

\begin{proposition}
\label{fdiff}
Given a function $f \colon \R^3\to \R$ such that $f(x, u, v)$ is $C^2$ in $(u, v)$ and piecewise continuous in $x$, it follows that $\mathcal{F} \colon H^2(I) \to L^2(I)$ defined by $\mathcal{F}(u) = u'' - f(x,u,u')$ is Frech\'{e}t differentiable and, for $w\in H^2(I)$, the derivative $D\mathcal{F}(w) \colon H^2(I)\to L^2(I)$ is given by
\begin{equation*} (D\mathcal{F}(w))(v) = v'' - f_u(x,w,w')v - f_{u'}(x,w,w')v', \quad \text{for all}~ v\in H^2(I).
\end{equation*}
\end{proposition}

\begin{proof}
Before proceeding with the proof, we should notice that from Lemma~\ref{lemmacont} it follows that
\begin{equation}
\label{c1est}
\left\|u\right\|_{C^1(\bar{I})}\leq c_1 \left\|u\right\|_{H^2(I)}\mbox{ for all }u\in H^2(I).
\end{equation}
Now let $u\in H^2(I)$ be fixed. Since $u\in H^2(I)$ implies that $u\in C^1(\bar{I})$, letting $z = (x,u,u')$, it follows that $f_u(z)$ and $f_{u'}(z)$ are piecewise continuous, and since $\left\|v\right\|_{L^2(I)}\leq \left\|v\right\|_{H^2(I)}$ and $\left\|v'\right\|_{L^2(I)}\leq \left\|v\right\|_{H^2(I)}$ it follows that
\begin{equation*} \left\|D\mathcal{F}(u)(v)\right\|_{L^2(I)}\leq \left(1+\sup_{x\in \bar{I}} \left|f_u(z)\right| + \sup_{x\in \bar{I}} \left|f_{u'}(z)\right|\right)\left\|v\right\|_{H^2(I)}
\end{equation*}
for all $v \in H^2(I)$ and thus $D\mathcal{F}(u) \colon H^2(I)\to L^2(I)$ is a bounded linear operator. 

Now, notice that $g^x \colon \R^2\to\R$ given by $g^x(u,v) = f(x,u,v)$ is $C^2$ for each fixed $x\in \bar{I}$ and that $D_{(u,v)}f(x,u,v)$ and $D_{(u,v)}^2f(x,u,v)$ are continuous in $(u, v)$ and piecewise continuous in $x$. Hence, letting $r = c_1 \left\| u \right\|_{H^2(I)} + c_1$ and
\[
L = \sup_{x \in \bar{I}, (u, v) \in B(0,r)_{\R^2}} \left\| D_{(u,v)}^2f(x,u,v) \right\|,
\]
it follows, by the Taylor Theorem (see \citep[Theorem 5.6.2]{1971-Cartan}), that
\begin{equation}
\label{taylorapl}
\begin{aligned}
|f(x, u_1, v_1) - f(x, u_0, v_0) - (D_{(u,v)}f(x, u_0, v_0))(h_0, k_0) | \leq\frac{L}{2} \|(h_0,k_0) \|_{\R^2}^2
\end{aligned}
\end{equation}
for all $(x,u_0,v_0), (x,u_1,v_1)\in \bar{I}\times \bar{B}(0,r)_{\R^2}$, where $h_0=u_1-u_0$ and $k_0=v_1-v_0$.

Thus, given $h\in B(0,1)_{H^2(I)}$ it follows from (\ref{c1est}) that $\left\|(u(x),u'(x))\right\|_{\R^2}\leq  c_1 \left\| u \right\|_{H^2(I)} \leq r$, and
\[
\left\|(u(x)+h(x),u'(x)+h'(x))\right\|_{\R^2} \leq c_1\left\| u \right\|_{H^2(I)} + c_1 = r
\]
for all $x\in \bar{I}$. Hence, denoting $w=u+h$, it follows by (\ref{taylorapl}) that
\begin{align*}
& |(\mathcal{F}(u+h) - \mathcal{F}(u) - D\mathcal{F}(u)(h))(x)| = \\
& |f(x,u(x),u'(x)) - f(x,w(x),w'(x)) - D_{(u,v)}f(x,u(x),u'(x))(h(x),h'(x)) | \\
& \leq \frac{L}{2} \|(h(x),h'(x)) \|_{\R^2}^2\leq c_1^2\frac{ L}{2}\|h \|_{H^2(I)}^2
\end{align*}
for all $x\in \bar{I}$. Finally, taking the $L^2$ norm and dividing by $\left\|h\right\|_{H^2(I)}$ we obtain
\begin{equation*}
\frac{\left\|\mathcal{F}(u+h)-\mathcal{F}(u)-D\mathcal{F}(u)(h)\right\|_{L^2(I)}}{\left\|h\right\|_{H^2(I)}} \leq c_1^2\frac{L}{2}\left\|h\right\|_{H^2(I)}
\end{equation*}
for all non-zero $h\in B(0,1)_{H^2(I)}$. Thus the left side in the above inequality goes to $0$ as $\left\|h\right\|_{H^2(I)}\to 0$, which concludes the proof.
\end{proof}

The next result is useful to compute $D\mathcal{F}(w)_{\cos,m} \colon \R^{m} \to \R^{m}$ in the examples. 
\begin{corollary}
\label{cor3} 
Under the hypothesis of Proposition~\ref{fdiff} we have   $\left(D\mathcal{F}(w\right)_{\cos,m})(a)=A a$ where that $A\in M_m(\R)$ is defined by
\begin{equation*} 
A_{1,j} = \int_0^1 u_j''(x) - f_u(x,w(x),w'(x))u_j(x) - f_{u'}(x,w(x),w'(x))u_j'(x) dx
\end{equation*}
and
\begin{equation*}
A_{i,j} = \int_0^1 \left(u_j''(x) - f_u(x,w(x),w'(x))u_j(x) - f_{u'}(x,w(x),w'(x))u_j'(x) \right)\sqrt{2}\cos((i-1)\pi x) dx,
\end{equation*}
where, $u_1 = 1$, $u_j(x) = \frac{\sqrt{2}\cos((j-1)\pi x)}{\omega(j)}$ for $1<i,j\leq m$.
\end{corollary}

\begin{proof}
Given $w\in C^2(\bar{I})$ the derivative $D\mathcal{F}(w)_{\cos,m} \colon \R^{m} \to \R^{m}$ can be expressed by $\left(D\mathcal{F}(w)_{\cos,m}\right)(a)=A a$, where the columns of $A$ are given by $D\mathcal{F}(w)_{\cos,m}(e_{j})$ for $1\leq j \leq m$. Since $h_{\cos,m}^{-1}(e_1) = 1$ and $h_{\cos,m}^{-1}(e_j) = \frac{\sqrt{2}\cos((j-1)\pi x)}{\omega(j)}$ for $1< j \leq m$, follows by Proposition~\ref{fdiff} that 
\begin{equation*} 
D\mathcal{F}(w)_{\cos,m}(e_{1}) = \pi_{\cos,m}\left(-f_{u}(x,w,w')\right) = \left(A_{1,1},\ldots, A_{m,1} \right)
\end{equation*}
and
\begin{equation*}
\begin{aligned}
D\mathcal{F}(w)_{\cos,m}(e_{j}) = & \pi_{\cos,m}\left(-\frac{(\pi (j-1))^{2}\sqrt{2}}{\omega(j)} \cos((j-1)\pi x)\right) - \\
& \pi_{\cos,m}\left(\frac{\sqrt{2}}{\omega(j)} \cos((j-1)\pi x) f_{u}(x,w,w')\right) + \\
& \pi_{\cos,m}\left(\frac{((j-1)\pi \sqrt{2}}{\omega(j)}\sin((j-1)\pi x) \right) = \left(A_{1,j},\ldots, A_{m,j} \right).\\
\end{aligned}
\end{equation*}
The coordinate-wise of the projection $\pi_{\cos,m}$ are the inner product with elements of orthonormal cosine basis in $L^{2}(I)$, that is,
\begin{equation*}
A_{1,1} = -\int_0^1 f_u(x,w(x),w'(x)) dx,
\end{equation*}

\begin{equation*}
\begin{aligned}
A_{1,j} = & -\frac{\sqrt{2}}{\omega(j)} \int_0^1  f_u(x,w(x),w'(x)) \cos((j-1)\pi x) dx \\ & + \frac{(j-1)\pi \sqrt{2}}{\omega(j)} \int_0^1  f_{u'}(x,w(x),w'(x)) \sin((j-1)\pi x) dx
\end{aligned}
\end{equation*}

and

\begin{equation*}
\begin{aligned}
A_{i,j} = & - \dfrac{((j-1)\pi)^{2}}{\omega(j)} \int_0^1  2 \cos((i-1)\pi x) \cos((j-1)\pi x)dx \\
& - \dfrac{2}{w(j)} \int_0^1  \cos((i-1)\pi x) f_u(x,w(x),w'(x)) \cos((j-1)\pi x)  dx \\
& + \frac{2(j-1)\pi}{\omega(j)} \int_0^1  \cos((i-1)\pi x)  f_{u'}(x,w(x),w'(x)) \sin((j-1)\pi x)dx \\
= & - \dfrac{((j-1)\pi)^{2}}{\omega(j)} \delta_{ij} - \dfrac{2}{w(j)} \int_0^1 \cos((i-1)\pi x) f_u(x,w(x),w'(x)) \cos((j-1)\pi x) dx \\
& + \frac{2(j-1)\pi}{\omega(j)} \int_0^1  \cos((i-1)\pi x)  f_{u'}(x,w(x),w'(x)) \sin((j-1)\pi x)dx.
\end{aligned}
\end{equation*}    
\end{proof}

We propose the definition of the bijectivity modulus below in order to formulate our theory of isolation of zeros in two-point boundary value problems. In what follows we always let $X$, $Y$ and $Z$ denote Banach Spaces.

\begin{definition} For $F\in \mathcal{L}(X,Y)$ we define the \textit{bijectivity} modulus $\lambda(F)$ of $F$ by

\begin{equation*}
\lambda(F) =
\begin{cases} \|F^{-1}\|^{-1}, & \text{if}~ F ~\text{is invertible} \\
0, & \text{otherwise}
\end{cases}
\end{equation*}
\end{definition}

Notice that $\lambda(F)$ is well defined since, due to the Open Mapping Theorem, the invertibility of a bounded linear operator $F \in \mathcal{L}(X,Y)$ implies directly in the boundedness of its inverse, and $F^{-1} \neq 0$ implies that $\| F^{-1} \| \neq 0$.

Before we proceed, we shall prove some basic properties of the bijectivity modulus, which will be used to prove the main theorem of this paper (Theorem~\ref{strong}).

\begin{proposition}
\label{basic2}
Given $F\in \mathcal{L}(X,Y)$ and $G\in \mathcal{L}(X,Y)$ it follows that $\left|\lambda(F)-\lambda(G)\right|\leq \left\|F-G\right\|$.
\end{proposition}

\begin{proof}
Let $j(F) = \inf_{\left\|x\right\|_X=1}\left\|Fx\right\|_Y$, called the injectivity modulus of $F$ and $k(F) = \sup \{r\geq 0\ |\ B(0,r)_Y\subset F(B(0,1)_X)\}$, called the surjectivity modulus of $F$ (see \citep[Definition 3, p. 86]{2007-Muller}). Following \citep[Theorem 4, p. 86, Theorem 7, p. 87 and Proposition 9, p. 88]{2007-Muller}, the following properties regarding the injectivity and surjectivity modulus are valid
\begin{enumerate}[$(a)$] \itemsep1em
\item $j(F)\neq 0$ if and only if $F$ is injective with range closed, and $k(F)\neq 0$ if and only if $F$ is surjective;
\item If $F$ is invertible then $j(F)=k(F)= \|F^{-1}\|^{-1}$;
\item $\left|j(F)-j(G)\right|\leq \left\|F - G\right\|$ and $\left|k(F)-k(G)\right|\leq \left\|F - G\right\|$ for all $F,G\in \mathcal{L}(X,Y)$.
\end{enumerate}
Notice that, if $F$ and $G$ are both not invertible, then by definition we have $\lambda(F)=\lambda(G)=0$ in which case (\ref{toprove}) is trivially true. Thus, we can suppose without loss of generality that $F$ is invertible and prove that the following holds in such case:
\begin{equation}
\label{toprove}
\left|\lambda(F)-\lambda(G)\right|=\left|j(F)-j(G)\right| \text{ or } \left|\lambda(F)-\lambda(G)\right|=\left|k(F)-k(G)\right|
\end{equation}
which, due to item $(c)$ above, will conclude the proof.

If $F$ and $G$ are both invertible, then by the definition of bijectivity modulus and due to item $(b)$ above it follows that
\begin{equation*}
\lambda(F) = \|F^{-1}\|^{-1}=j(F)\mbox{ and }\lambda(G)= \|G^{-1}\|^{-1}=j(G)
\end{equation*}
in which case (\ref{toprove}) is true.
Now, if we suppose that $F$ is invertible and $G$ is not invertible then it would follow that either $G$ is not injective or not onto. In case $G$ is not injective then by items $(a)$ and $(b)$ and by the definition of bijectivity modulus we have
\begin{equation*}\lambda(F)=\left\|F^{-1}\right\|^{-1}=j(F)\mbox{ and }\lambda(G)=0=j(G)
\end{equation*}
in which case (\ref{toprove}) is true.  Analogously, if $F$ is invertible and $G$ is not onto then  $\lambda(F)=\left\|F^{-1}\right\|^{-1}=k(F)$ and $\lambda(G)=0=k(G)$, in which case (\ref{toprove}) is true as well.
\end{proof}

The next definitions are usual in operator theory (see \citep[Sec. 5.4]{2011-Brezis} and \citep[Sec. 2.6]{1983-Marti}). 

\begin{definition}
We say $X=X_1\oplus X_2$ is a \emph{Hilbert sum} if $X$ is a Hilbert space written as the direct sum of $X_1$ and $X_2$, closed subspaces of $X$, such that $\langle x_1,x_2 \rangle = 0$ for all $x_1 \in X_1$ and $x_2\in X_2$.
\end{definition}

\begin{definition}
If $X=X_1\oplus X_2$ is a Hilbert sum then, given $F_1 \colon X_1\to X_1$ and $F_2 \colon X_2\to X_2$, we let $F=F_1\oplus F_2 \colon X \to X$ be the function defined by $F(x)=F_1(x_1)+F_2(x_2)$ for $x=x_1+x_2\in X$ with $x_1\in X_1$ and $x_2\in X_2$.
\end{definition}

\begin{proposition}
\label{bij_mod}
Let $X=X_1\oplus X_2$ be a Hilbert sum and given linear operators $F_1 \colon X_1\to X_1$ and $F_2 \colon X_2\to X_2$, suppose that $F=F_1\oplus F_2$ is bounded. Then $F_1$ and $F_2$ are bounded and
\begin{itemize}
\item[(i)] $\left\|F\right\|=\max\left\{\left\|F_{1}\right\|,\left\|F_{2}\right\|\right\}$;
\item[(ii)] $\lambda(F) = \min \left\{\lambda(F_{1}), \lambda(F_{2})\right\}$.
\end{itemize}
\end{proposition}

\begin{proof}
Since $F_1$ and $F_2$ are the restrictions of $F$ to $X_1$ and $X_2$, respectively, it follows directly that $F_1$ and $F_2$ are bounded. On the other hand $(i)$ is a classical property in operator theory (see  \citep[p. 122]{1997-Kadison}).

For item $(ii)$, it is a classical property in operator theory that $F=F_1\oplus F_2$ is invertible if and only if both $F_1$ and $F_2$ are invertible, in which case $F^{-1}=F_1^{-1}\oplus F_2^{-1}$ (see \citep[Problem 6.1.17]{2002-Abramovich}). 

Thus, if either $F_1$ or $F_2$ is not invertible then, by the above property, $F$ is not invertible as well, in which case the definition of the bijectivity modulus implies that
\begin{equation*}
\min\{ \lambda(F_{1}), \lambda(F_{2}) \} = 0 = \lambda(F)
\end{equation*}
and the proposition is true in this case. On the other hand, if both $F_1$ and $F_2$ are invertible, then by the property stated above it follows that $F$ is invertible as well with $F^{-1}=F_1^{-1} \oplus F_2^{-1}$ and thus the definition of bijectivity modulus and item $(i)$ implies that
\begin{equation*}
\begin{aligned}
\lambda(F) & = \|F^{-1}\|^{-1} = \left( \max \left\{ \|F_{1}^{-1}\|, \|F_{2}^{-1}\|\right\} \right)^{-1} \\
& = \min \left\{ \|F_{1}^{-1}\|^{-1}, \|F_{2}^{-1}\|^{-1}\right\}
= \min \left\{ \lambda(F_{1}), \lambda(F_{2}) \right\}.
\end{aligned}
\end{equation*}
which proves the proposition.
\end{proof}

\section{Isolation of zeros for BVP with Neumann boundary conditions}
\label{sec:Isolation_zeros}

Letting $\mathcal{F} \colon H^2_N(I)\to L^2(I)$ be defined by
\begin{equation}
\label{eqf}
\mathcal{F}(u) = u'' - f(x,u,u')
\end{equation}
and given an approximate zero $w\in C^2(I)$ for $\mathcal{F}(u)=0$, our objective in this section is to compute the constants $\eta$, $\nu$ and $K$ needed to apply the Kantorovich Theorem (Theorem \ref{kantorovich2} bellow) in order to check for zeros of $\mathcal{F}$ near $w$.

\subsection{Kantorovich Theorem and Computation of the Bounds}

We now present the main theorems used by our method. The proofs of these results are given in the next subsections. We provide the following reformulation of the Kantorovich Theorem, using the bijectivity modulus.

\begin{theorem}
\label{kantorovich2}
Let $U\subset X$ be open, $\mathcal{F} \colon U\subset X \to Y$ differentiable, $A \subset U$ be a non-empty convex domain, $x_0\in A$, $R>0$ satisfying $\bar{B}(x_0,R)\subset A$, and $\eta, \nu, K\in \R$ satisfying:

\begin{enumerate}[(i)]
\setlength\itemsep{1em}
\item $\|\mathcal{F}(x_0)\|\leq \eta$, $\lambda(D\mathcal{F}(x_0))\geq \nu$ and $\left\|D\mathcal{F}(x)-D\mathcal{F}(y)\right\|\leq K\left\|x-y\right\|$ for all $x, y\in A$.
\item $g_1(t)=\eta- \nu t+\dfrac{K}{2}t^2$ has zeros $t^*<t^{**}$ with $t^*\in [0,R]$.
\end{enumerate}
Then $\mathcal{F}$ has a zero $x^*\in \bar{B}(x_0,t^*)$ and no other zeros in $B(x_0,t^{**})\cap A$.
\end{theorem}

\begin{proof}
Notice that if $\nu=0$ then $g_1$ could not possibly have distinct non-negative zeros. Hence it follows that $\nu>0$ and thus $\lambda(D\mathcal{F}(x_0))\geq \nu>0$, which by definition implies that $D\mathcal{F}(x_0)$ is invertible with $\|D\mathcal{F}(x_0)^{-1}\|=\lambda(D\mathcal{F}(x_0))^{-1}\leq \nu^{-1}$. Thus, letting $\eta^*=\nu^{-1}\eta$ and $K^*=\nu^{-1}K$ it follows that
\begin{equation*}
\begin{aligned}
\|D\mathcal{F}(x_0)^{-1}\mathcal{F}(x_0)\|\leq \|D\mathcal{F}(x_0)^{-1}\| \| \mathcal{F}(x_0)\|\leq \eta^*
\end{aligned}
\end{equation*}
and
\begin{equation*}
\begin{aligned}
\|D\mathcal{F}(x_0)^{-1}(D\mathcal{F}(x)-D\mathcal{F}(y))\|  \leq \|D\mathcal{F}(x_0)^{-1}\| \|(D\mathcal{F}(x)-D\mathcal{F}(y))\| \leq K^*\|x-y\|
\end{aligned}
\end{equation*}
for all $x,y\in A$. It then follows from item $(ii)$ that the polynomial $g(t)=\eta^* - t + \dfrac{K^*}{2}t^2=\nu^{-1}g_1(t)$ has zeros $t^*<t^{**}$ with $t^*\in [0,R]$, and thus we conclude from the Kantorovich Theorem (see \citep{1982-Kantorovich,2017-Fernandez}) that $\mathcal{F}$ must have a zero $x^*\in \bar{B}(x_0,t^*)$ and no other zeros in $B(x_0,t^{**})\cap A$, proving the theorem. 
\end{proof}

The following theorem tells us how to compute the constants $\nu$ and $K$ in item $(i)$ above, when $\mathcal{F} \colon H_N^2(I)\to L^2(I)$ is given by (\ref{eqf}). Let $\left\|g\right\|_{C^0(\bar{I})}=\sup_{x\in \bar{I}}\left|g(x)\right|$ for $g\in C^0(\bar{I})$, $\left\|g\right\|_{C^1(\bar{I})}=\sup_{x\in \bar{I}}\sqrt{g(x)^2+g'(x)^2}$ for $g\in C^1(\bar{I})$, and let $c_1=\left(\tanh(1)\right)^{-\frac{1}{2}}=\sqrt{\frac{e^2+1}{e^2-1}}$. The proof of this theorem is presented in the following subsections.

\begin{theorem}
\label{strong}
Let $f \colon \R^3\to \R$ be such that $f(x, u, v)$ is $C^2$ in $(u, v)$ and piecewise continuous in $x$, $w \colon \bar{I} \to \R$ be a $C^2$ function in $H^2_N(\bar{I})$, let $z(x)=(x,w(x),w'(x))$ and suppose that
\begin{equation}
\label{eq:estimateN}
\left\|f_u(z)\right\|_{C^0(I)}+\left\|f_u(z)\right\|_{C^1(I)}+\left\|f_{u'}(z)\right\|_{C^0(I)}+\left\|f_{u'}(z)\right\|_{C^1(I)} \leq N,
\end{equation}
and
\begin{equation*}
c_1\left\|D_{(u,v)}^2f(x,u,v)\right\|\leq K \quad \text{for all} \quad (x,u,v) \in \bar{I} \times \bar{B}\left(0,c_1 r\right)_{\R^{2}}.
\end{equation*}
Then $\mathcal{F} \colon H^2_N(I)\to L^2(I)$, as defined in (\ref{eqf}), is Frech\'{e}t-differentiable and moreover
\begin{enumerate}[$(i)$]
\item $\lambda\left(D\mathcal{F}(w)\right)\in\left[L-\frac{N}{\pi m},L+\frac{N}{\pi m}\right]$, where $L=\min \left\{\lambda\left(D\mathcal{F}(w)_{\cos, m} \right), \dfrac{(\pi m)^2}{\omega(m+1)}\right\}$;
\item $\|D\mathcal{F}(u_1)-D\mathcal{F}(u_2)\|\leq  K\|u_1-u_2 \|_{H^2_N(I)}$ for all $u_1,\, u_2\in \bar{B}\left(0,r\right)_{H^2(I)}$.
\end{enumerate}
\end{theorem}

\subsection{Proof of item $(i)$ of Theorem~\ref{strong}}

\begin{definition}
Given $m\geq 1$, we define the truncation operators $\rho_{\cos,m}, \rho_{\sin,m} \colon L^2(I)\to C^\infty(\bar{I})$ and the defect operators $\rho^*_{\cos,m}, \rho^{*}_{\sin,m} \colon L^2(I)\to L^2(I)$ by
\begin{equation*}
\begin{aligned}
\rho_{\cos,m}(u) = \widehat{u}_{\cos}(1)+\sum_{k=2}^{m} \widehat{u}_{\cos}(k)\sqrt{2}\cos((k-1)\pi x),\quad \rho^*_{\cos,m}(u) = u-\rho_{\cos,m}, \\
\rho_{\sin,m}(u) = \sum_{k=2}^{m} \widehat{u}_{\sin}(k)\sqrt{2}\sin((k-1)\pi x), \quad \text{and} \quad \rho^*_{\sin,m}(u) = u-\rho_{\sin,m}.
\end{aligned}
\end{equation*}
\end{definition}

\begin{proposition}
\label{lemma1}
Given $u\in H^1(I)$ and $m\geq 1$ it follows that
\begin{enumerate}[$(i)$]
\item $\left\|\rho^*_{\sin,m}(u)\right\|_{L^2(I)} \leq \dfrac{1}{\pi m} \|u'\|_{L^2(I)}$ \; for $u\in H^1_0(I)$;
\item $\left\|\rho^*_{\cos,m}(u)\right\|_{L^2(I)} \leq \dfrac{1}{\pi m} \|u'\|_{L^2(I)}$;
\item $\left\|\rho_{\cos,m}(u)\right\|_{H^1(I)} \leq \|u\|_{H^1(I)}$.
\end{enumerate}
\end{proposition}

\begin{proof}
Item $(i)$: Given $u\in H^1_0(I)$, since $u' \in L^2(I)$, due to Corollary~\ref{parsevalcos} we have $\widehat{u'}_{\cos} := \pi_{\cos}(u') \in \ell^2(\N)$ and $\pi_{\sin} \left (\rho^*_{\sin,m}(u)\right) \in \ell^2(\N)$ with $\|\widehat{u'}_{\cos}\|_{\ell^2(\N)}=\|u'\|_{L^2(I)}$ and $\left \|\pi_{\sin} \left (\rho^*_{\sin,m}(u)\right)\right \|_{\ell^2(\N)} = \|\rho^*_{\sin,m}(u)\|_{L^2(I)}$, for $m \in \N$. Furthermore, from Proposition~\ref{lemma0} we can write $\widehat{u}_{\sin}(k)=\dfrac{1}{(k-1)\pi}\widehat{u'}_{\cos}(k)$ for all $k\in \N$, $k>2$. Thus, item $(i)$ follows from
\begin{equation*}
\begin{aligned}
\|\rho^*_{\sin,m}(u) \|^{2}_{L^2(I)} & = \left \|\pi_{\sin} \left (\rho^*_{\sin,m}(u)\right)\right\|^{2}_{\ell^2(\N)} = \sum_{k=m+1}^\infty \left(\frac{1}{\pi (k-1)}\widehat{u'}_{\cos}(k)\right)^2 \\
& \leq \frac{1}{\pi^{2} m^{2}} \sum_{k=m+1}^\infty \left( \widehat{u'}_{\cos}(k)\right)^2 \leq \frac{1}{\pi^{2} m^{2}} \|\widehat{u'}_{\cos} \|^{2}_{\ell^2(\N)} .
\end{aligned}
\end{equation*}
The proof of item $(ii)$ is analogous.

\noindent
Item $(iii)$: Since due to item $(i)$ of Proposition~\ref{lemma0} we have
\begin{equation*}
\left(\rho_{\cos,m}(u)\right)'=\sum_{k=2}^{m}\widehat{u}_{\cos}(k)\left(- (k-1)\pi \sqrt{2}\sin((k-1)\pi x)\right)=\sum_{k=2}^{m}\widehat{u'}_{\sin}(k)\sqrt{2} \sin((k-1)\pi x)
\end{equation*}
and since $\left\|v\right\|^2_{L^2(I)} = \left\|\pi_{\sin}(v)\right\|^2_{\ell^2(\N)} = \sum_{k=2}^\infty\widehat{v}_{\sin}(k)^2$ for all $v\in L^2(I)$ (see Corollary~\ref{parsevalcos}), it follows that
\begin{equation*} \left\|\left(\rho_{\cos,m}(u)\right)'\right\|^2_{L^2(I)} = \sum_{k=2}^{m} \widehat{u'}_{\sin}(k)^2 \leq \sum_{k=2}^\infty \widehat{u'}_{\sin}(k)^2 = \left\|u'\right\|^2_{L^2(I)}.
\end{equation*}
Analogously we can show that $\left\|\rho_{\cos,m}(u)\right\|_{L^2(I)}\leq \left\|u\right\|_{L^2(I)}$. Therefore it follows that
\begin{equation*}
\left\|\rho_{\cos,m}(u)\right\|^2_{H^1(I)} = \left\|\rho_{\cos,m}(u)\right\|_{L^2(I)}^2 + \left\|\left(\rho_{\cos,m}(u)\right)'\right\|_{L^2(I)}^2 \leq \left\|u\right\|_{L^2(I)}^2 + \left\|u'\right\|_{L^2(I)}^2 = \left\|u\right\|^2_{H^1(I)}
\end{equation*}
which proves $(iii)$.
\end{proof}

\begin{lemma}
\label{lemma3}
Let $g \colon \R^3\to \R$ be a $C^1$ function and $w \colon \bar{I} \to \R$ be a $C^2$ function. Letting $z(x)=(x,w(x),w'(x))$ for $x \in \bar{I}$, suppose that $\left\|g(z)\right\|_{C^0(I)}+\left\|g(z)\right\|_{C^1(I)} \leq N$, then
\begin{enumerate}[$(i)$]
\itemsep1em
\item $\|g(z)v - \rho_{\cos,m}\left(g(z)\rho_{\cos,m}(v)\right)\|_{L^2(I)} \leq \dfrac{N}{\pi m} \left\|v\right\|_{H^1(I)}$ ~~for all $v\in H^1(I)$;
\item $\|g(z)v - \rho_{\cos,m}\left(g(z)\rho_{\sin,m}(v)\right)\|_{L^2(I)}\leq \dfrac{N}{\pi m} \left\|v\right\|_{H^1(I)}$ ~~for all $v\in H^1_0(I)$.
\end{enumerate}
\end{lemma}

\begin{proof}
Item $(i)$: Given $v\in H^1(I)$, since
\begin{equation}
\label{2thesisItem}
g(z)v - \rho_{\cos,m}(g(z)\rho_{\cos,m}(v)) = g(z)\, \rho^*_{\cos,m}(v) + \rho^*_{\cos,m}(g(z)\, \rho_{\cos,m}(v)),
\end{equation}
we just need to estimate the norms of the last two terms in order to prove item $(i)$. To estimate $g(z)\, \rho^*_{\cos,m}(v)$, notice that, due to item $(ii)$ of Proposition~\ref{lemma1}, given $v\in H^1(I)$ and $m\in \N$, we have
\begin{equation}
\label{lemma3eq1}
\begin{aligned}
\| g(z)\rho^*_{\cos,m}(v)\|_{L^2(I)} & \leq \|g(z)\|_{C^0(\bar{I})}\|\rho^*_{\cos,m}(v)\|_{L^2(I)} \leq
\frac{1}{\pi m} \left\|g(z)\right\|_{C^0(\bar{I})}\|v'\|_{L^2(I)} \\
& \leq \frac{1}{\pi m} \left\|g(z)\right\|_{C^0(\bar{I})}\|v\|_{H^1(I)}.
\end{aligned}
\end{equation}
Now, to estimate $\|\rho^*_{\cos,m}(g(z)\rho_{\cos,m}(v))\|_{L^2(I)}$ we first compute $\left\|\left(g(z)\, \rho_{\cos,m}(v)\right)'\right\|_{L^2(I)}$. Since $g(z)\in C^1(\bar{I})$ and $\rho_{\cos,m}(v)\in C^{\infty}(\bar{I})$ it follows that $g(z)\, \rho_{\cos,m}(v)\in C^1(\bar{I})$, and thus from the Cauchy-Schwarz inequality it follows that
\begin{equation*}
\begin{aligned}
& |(g(z) \rho_{\cos,m}(v))'(x)| \leq \left|(g(z(x)))' \right| \left|(\rho_{\cos,m}(v))(x)\right|+ \left|g(z(x)) \right|\left|(\rho_{\sin,m}(v'))(x)\right| \leq \\
& \sqrt{\left(\left(g(z(x))\right)'\right)^2 + \left(g(z(x))\right)^2}\sqrt{\left((\rho_{\cos,m}(v))(x)\right)^2 + \left((\rho_{\sin,m}(v'))(x)\right)^2} \leq \\
& \left\|g(z)\right\|_{C^1(\bar{I})}\sqrt{\left((\rho_{\cos,m}(v))(x)\right)^2 + \left((\rho_{\sin,m}(v'))(x)\right)^2}
\end{aligned}
\end{equation*}
for all $x\in \bar{I}$ and $m\in \N$. Note that $\rho_{\sin,m}(v') = \left(\rho_{\cos,m}(v)\right)'$, hence from Proposition~\ref{lemma1} and taking the $L^2(I)$ norm in the above inequality we obtain
\begin{equation*}
\left\|\left(g(z)\, \rho_{\cos,m}(v)\right)'\right\|_{L^2(I)}\leq \left\|g(z)\right\|_{C^1(\bar{I})}\left\|\rho_{\cos,m}(v)\right\|_{H^1(I)}\leq \left\|g(z)\right\|_{C^1(\bar{I})}\left\|v\right\|_{H^1(I)}.
\end{equation*}
Therefore, by item $(ii)$ of Proposition~\ref{lemma1} and the inequality above, it follows that
\begin{equation}
\label{lemma3eq3}
\begin{aligned}
\|\rho^*_{\cos,m}(g(z)\, \rho_{\cos,m}(v))\|_{L^2(I)} & \leq \frac{1}{\pi m} \left\|\left(g(z)\, \rho_{\cos,m}(v)\right)'\right\|_{L^2(I)} \\
& \leq \frac{1}{\pi m} \left\|g(z)\right\|_{C^1(\bar{I})} \|v \|_{H^1(I)}.
\end{aligned}
\end{equation}
Thus, combining (\ref{2thesisItem}), (\ref{lemma3eq1}), and (\ref{lemma3eq3}), it follows that
\begin{equation*}
\begin{aligned}
& \|g(z)v - \rho_{\cos,m}(g(z)\rho_{\cos,m}(v))\| \leq \|g(z)\, \rho^*_{\cos,m}(v)\| + \|\rho^*_{\cos,m}(g(z)\, \rho_{\cos,m}(v))\|
\leq \\
& \frac{1}{\pi m} \left\|g(z)\right\|_{C^0(\bar{I})}\|v\|_{H^1(I)} + \frac{1}{\pi m} \left\|g(z)\right\|_{C^1(\bar{I})} \|v \|_{H^1(I)} \leq \dfrac{N}{\pi m} \left\|v\right\|_{H^1(I)},
\end{aligned}
\end{equation*}
which proves $(i)$. The proof of item $(ii)$ is analogous.
\end{proof}

Lemma~\ref{bij_mod_corol} below is the bridge to connect our truncation estimates in order to compute bounds for $\lambda(\mathcal{F})$. Before we state the result, recall that we identify $\R^m$ with a subset of $\ell^2(\N)$ through the isometric embedding $\pi_{\R^m} \colon \R^m \to \ell^2(\N)$ given by $\pi_{\R^m}(a)=(a_1,\cdots,a_m, 0, \cdots)\in \ell^2(\N)$. We also define the projection operators $\pr_m$ and $\pr_{m}^*$ below.

\begin{definition}
Given $m\geq 1$, we define the projection operators $\pr_m \colon \ell^2(\N)\to \R^m$ and $\pr_m^* \colon \ell^2(\N) \to (\R^m)^*$ by
\begin{equation*}
\begin{aligned}
\pr_m(a) = \left (a(1),\cdots,a(m)\right),\\
\pr^*_m(a) = a - \pi_{\R^m} (\pr_m(a)),\\
\end{aligned}
\end{equation*}
where $(\R^m)^*$ is the Hilbert space $\{ a \in \ell^2(\N) \mid a(1) = \cdots = a(m) = 0\}$ with inner product inherited from $\ell^2(\N)$.
\end{definition}

\begin{lemma}
\label{bij_mod_corol}
Let $m>1$ and $(a_{k})_{k\in \N}$ be such that $\lim_{k \to +\infty} a_{k} = r \neq 0$ and $a_k\neq 0$ for $k\geq m+1$. Given a bounded linear operator $\mathcal{P} \colon \ell^2(\N)\to \ell^2(\N)$ let $G \colon \ell^2(\N) \to \ell^2(\N)$ be defined by $(G(v))_k = a_k v_k - (\mathcal{P}(v))_k$, and let
\begin{enumerate}[$(i)$]
\item $\mathcal{P}_{m} = \pr_m \: \circ \: \mathcal{P} \: \circ \: \pi_{\R^m} \colon \R^m\to \R^m$;
\item $G_m = \pr_m \; \circ \: G \; \circ \: \pr_m \colon \ell^2(\N) \to \R^m\subset \ell^2(\N)$.
\end{enumerate}
Then $G$, $\mathcal{P}_m$ and $G_m$ are bounded and $\lambda(\mathcal{P}) \in \left[L - \epsilon_m, L + \epsilon_m \right]$ where
\begin{equation*}
L = \min{\left\{\lambda(\mathcal{P}_{m}), \; \inf_{k\geq m+1} |a_{k}|\right\}}
\quad \text{and} \quad \epsilon_m = \left\|G-G_m\right\| .
\end{equation*}
\end{lemma}

\begin{proof}
Consider the Hilbert sum $\ell^2(\N)=\R^m\oplus (\R^m)^*$, and let
\begin{itemize}
\item[$\bullet$] $\eta \colon \ell^2(\N)\to \ell^2(\N)$ be defined by $(\eta(v))(k) = a_k v(k)$;
\item[$\bullet$] $\eta^* \colon (\R^m)^*\to (\R^m)^*$ be the restriction of $\eta$ to $(\R^m)^*$;
\item[$\bullet$] $\mathcal{P}_m^* \colon \ell^2(\N)\to \ell^2(\N)$ be defined by $\mathcal{P}^*_m=\mathcal{P}_m \oplus \eta^*$.
\end{itemize}
Since $\lim_{k\to \infty} a_k = r$, it follows that $(a_k)_{k\in \N}$ is a bounded sequence and thus $\eta$ and $\eta^*$ are bounded linear operators. Moreover, $G= \eta - \mathcal{P}$ is bounded as well, since $\mathcal{P}$ and $\eta$ are bounded. Similarly, note that $\pr_m \colon \ell^2(\N)\to \R^m$ is bounded, and hence $G_m = \pr_m \: \circ \: G \: \circ \: \pr_m$ is also bounded. On the other hand, given $v \in \ell^2(\N)$ we have
\begin{equation*} 
\begin{aligned}
\mathcal{P}_m^*(v)
& = \mathcal{P}_{m}(\pr_{m}(v)) + \eta^*(\pr^{*}_{m}(v))
= \pr_{m} \left( \eta(\pr_m(v))-G(\pr_m(v)) \right) + \eta^*(\pr^*_m(v)) \\
& = \eta(\pr_{m}(v)) + \eta^{*}(\pr^{*}_{m}(v)) - \pr_{m}(G(\pr_{m}(v)) \\
& = \eta(\pr_{m}(v)) + \eta(\pr^{*}_{m}(v)) - \pr_{m}(G(\pr_{m}(v))=\eta(v) -G_m(v)
\end{aligned}
\end{equation*}
where we identify $\pr_{m}(\eta(\pr_m(v)) \in \R^m$ with
$\eta(\pr_{m}(v)) \in \ell^2(\N)$ and note that $\eta^{*}(\pr^{*}_{m}(v)) = \eta(\pr^{*}_{m}(v))$ for all $v \in \ell^2(\N)$. Therefore $\mathcal{P}^*_m=\eta-G_m$ and, in particular, it follows that $\mathcal{P}^*_m$ is bounded and $\mathcal{P} - \mathcal{P}^*_m = G_m - G$. From Proposition~\ref{bij_mod} and $\mathcal{P}^*_m=\mathcal{P}_m \oplus  \eta^*$, we obtain
\begin{equation*}
\lambda(\mathcal{P}_m^*) = \min\left\{ \lambda(\mathcal{P}_{m}), \; \lambda(\eta^*)\right\}.
\end{equation*}
Combining these results with Proposition~\ref{basic2}, it follows that
\begin{equation}
\label{lambda_inq1}
\left|\lambda(\mathcal{P}) - \min \left\{\lambda(\mathcal{P}_{m}), \; \lambda(\eta^*)\right\} \right| = \left| \lambda(\mathcal{P}) - \lambda(\mathcal{P}_m^*)\right| \leq  \left\|\mathcal{P}-\mathcal{P}_m^*\right\| = \| G_{m} - G \|.
\end{equation}
Finally, since $\lim_{k \to +\infty} a_{k} = r \neq 0$ with $a_k\neq 0$ for $k\geq m+1$, the sequence $\left((a_k)^{-1}\right)_{k\geq m+1}$ is bounded and thus $\eta^* \colon (\R^m)^*\to (\R^m)^*$ is invertible with $\left\|\left(\eta^*\right)^{-1}\right\|_{\left(\R\right)^*}=\sup_{k\geq m+1}\left|\left({a_k}\right)^{-1}\right|$ and hence $\lambda(\eta^*)=\left\|(\eta^*)^{-1}\right\|_{\left(\R\right)^*}^{-1}=\inf_{k\geq m+1} \left|a_k\right|$. These facts combined with inequality (\ref{lambda_inq1}) concludes the proof.
\end{proof}

Now we are ready to prove the item $(i)$ of Theorem~\ref{strong}.

\begin{proof}[Proof of Theorem~\ref{strong} (i)]
Let $G \colon \ell^2(\N)\to \ell^2(\N)$ be defined by
\[
G(b) = \pi_{\cos}\left( f_u(z)h_{\cos}^{-1}(b) + f_{u'}(z)\left(h_{\cos}^{-1}(b)\right)'\right)
\]
and let
\begin{itemize}
\item[$\bullet$] $\mathcal{P}=\pi_{\cos} \:\circ\: D\mathcal{F}\left(w\right)\:\circ\: h^{-1}_{\cos\ } \colon \ell^2(\N)\to \ell^2(\N)$;
\item[$\bullet$] $\mathcal{P}_m = \pr_m \: \circ \: \mathcal{P} \: \circ \: \pi_{\R^m} = \pr_m \: \circ \: \mathcal{P} \mid_{\R^m} \colon \R^m \to \R^m$;
\item[$\bullet$] $G_m = \pr_m\:\circ\: G\:\circ\: \pr_m \colon \ell^2(\N)\to \R^m\subset  \ell^2(\N)$;
\item[$\bullet$] $a_1=0$ and $a_k = -\dfrac{ (\pi (k-1))^2}{\omega(k)}$ for $k\geq 2$.
\end{itemize}
Since $D\mathcal{F}\left(w\right)$ is a bounded linear operator (by Proposition~\ref{fdiff}) and $\pi_{\cos} \colon L^2(I)\to \ell^2(\N)$ and $h^{-1}_{\cos} \colon \ell^2(\N) \to H^2_N(I)$ are isometric isomorphisms, it follows that $\mathcal{P}=\pi_{\cos}\:\circ\: D\mathcal{F}\left(w\right)\:\circ\: h^{-1}_{\cos\ }$ is a bounded linear operator, with
$\lambda(\mathcal{P})=\lambda(D\mathcal{F}(w))$. Moreover, since $\pr_m \:\circ\: \pi_{\cos}=\pi_{\cos,m}$ it follows that 
\[
\mathcal{P}_m = \pr_m \:\circ \:\mathcal{P} \mid_{\R^m} = \pi_{\cos,m}\:\circ \:D\mathcal{F}(w)\:\circ \:h_{\cos}^{-1} \mid_{\R^m} = D\mathcal{F}(w)_{\cos,m}.
\]
On the other hand given $b\in \ell^2(\N)$, by Proposition~\ref{fdiff}, we have
\[
\begin{aligned}
\mathcal{P}(b) & = \pi_{\cos}\left((D\mathcal{F}(w))(h_{\cos}^{-1}(b))\right)
= \pi_{\cos}\left( \left(h_{\cos}^{-1}(b)\right)'' - f_u(z) h_{\cos}^{-1}(b) - f_{u'}(z) \left(h_{\cos}^{-1}(b)\right)' \right) \\
& = \pi_{\cos}\left(\left(h_{\cos}^{-1}(b)\right)''\right) - G(b),
\end{aligned}
\]
which can be written as
\[
(\mathcal{P}(b))_k = a_k b_k - (G(b))_k \quad \text{for all} \quad k\in \N.
\]
From Lemma~\ref{bij_mod_corol} the operators $G$, $\mathcal{P}_m$ and $G_m$ are bounded and $\lambda(D\mathcal{F}(w))=\lambda(\mathcal{P}) \in [L - \epsilon_m, L + \epsilon_m]$ where
\begin{equation*}
L = \min \left\{ \lambda(D\mathcal{F}(w)_{\cos,m}), \inf_{k\geq m+1} |a_{k}| \right\} \quad \text{and} \quad \epsilon_m = \left\|G-G_m\right\|_{\mathcal{L}(\ell^2(\N))}.
\end{equation*}

On the other hand, notice that $(-a_k)_{k\geq 2}$ is strictly increasing since $(-a_k)^{-1} = \sqrt{(\pi (k-1))^{-4}+(\pi (k-1))^{-2}+1}$ is strictly decreasing in $k$ for $k\geq 2$. Thus $\inf_{k\geq m+1}\left|a_k\right|=\left|a_{m+1}\right| = \dfrac{(\pi m)^2}{\omega(m+1)}$. Therefore, to prove item $(i)$ we just need to prove that $\epsilon_m \leq \frac{N}{\pi m}$.

For this end, letting $v(b)=h_{\cos}^{-1}(b)\in H_N^2(I)$ notice that since $\pi_{\cos}$ is an isometric isomorphism and since $v(b)'\in H^{1}_{0}(I)$, due to items $(i)$ and $(ii)$ of Lemma~\ref{lemma3} we have
\begin{align*}
& \left\|(G - G_m)(b) \right\|_{\ell^2(\N)} \leq \|f_{u'}(z)v(b)'-\rho_{\cos,m}(f_{u'}(z)v(b)') \|_{L^2(I)} + \\
& \|f_{u}(z)v(b)- \rho_{\cos,m}(f_{u}(z)v(b)) \|_{L^2(I)} \leq \frac{N \|v(b) \|_{H^2_N(I)}}{\pi m} = \frac{N}{\pi m}\|b \|_{\ell^2(\N)}
\end{align*}
for all $b \in \ell^2(\N)$, and hence
\[
\epsilon_m = \|G - G_m \|_{\mathcal{L}(\ell^2(\N))}\leq \frac{N}{\pi m}
\]
concluding the proof.
\end{proof}

\subsection{Proof of item $(ii)$ of Theorem~\ref{strong}}

\begin{lemma}
\label{lemmacont}
Consider $u\in H^1(I)$ and let $c_1=\left(\tanh(1)\right)^{-\frac{1}{2}}=\sqrt{\frac{e^2+1}{e^2-1}}$. Then
\begin{enumerate}[$(i$)]
\itemsep1em
\item $\left\|u\right\|_{C^0(\bar{I})} \leq c_1\left\|u\right\|_{H^1(I)}$;
\item $\left\|u\right\|_{C^0(\bar{I})} \leq \left\|u'\right\|_{L^2(I)}$ if $u\in H^1_0(I)$;
\item $\left\| u\right\|_{C^1(\bar{I})}\leq c_1\left\|u\right\|_{H^2(I)}$ if $u\in H^2_N(I)$.
\end{enumerate}
\end{lemma}

\begin{proof}
The proof of item $(i)$ can be found for instance in \citep{1983-Marti}, where it is shown that $c_1$ is the smallest constant for such inequality. To prove item $(ii)$ notice that since $u\in H^1_0(I)$ and $u(0)=0$, given $x\in \bar{I}$ it follows from the Fundamental Theorem of Calculus for $H^1$ functions (see \citep[Theorem 8.2]{2011-Brezis}) that
\begin{equation*}
u(x) = \int_0^x u'(t) dt \quad \Rightarrow \quad \left|u(x)\right|\leq \int_0^1 \left|u'(t)\right| dx \leq \sqrt{\int_0^1 u'(t)^2 dt} = \left\|u'\right\|_{L^2(I)}
\end{equation*}
which proves item $(ii)$. Finally, to prove item $(iii)$, given $x\in \bar{I}$, since $u'\in H^1_0(I)$ and $c_1\geq 1$, applying items $(i)$ and $(ii)$ we have
\begin{equation*}
\begin{aligned}
& \sqrt{u(x)^2 + u'(x)^2} \leq \sqrt{\left(c_1\left\|u\right\|_{H^1(I)}\right)^2 + \left\|u''\right\|_{L^2(I)}^2} \leq \\
& \sqrt{\left(c_1\left\|u\right\|_{H^1(I)}\right)^2 + \left(c_1\left\|u''\right\|_{L^2(I)}\right)^2} = c_1\left\|u\right\|_{H^2(I)}
\end{aligned}
\end{equation*}
concluding the proof.
\end{proof}

Now, we are ready to prove item $(ii)$ of Theorem~\ref{strong}.

\begin{proof}[Proof of Theorem~\ref{strong} (ii)]
Consider $u,\, w\in\bar{B}\left(0,r\right)_{H^2_N(I)}$ and $z\in H^2_N(I)$. We have by definition of $\mathcal{DF}$ that
\[
\left( (D\mathcal{F}(u)-D\mathcal{F}(w))(z) \right) (x)=(Df_{(u,v)}(x,u(x),u'(x))-Df_{(u,v)}(x,w(x),w'(x)))(z(x),z'(x)).
\]
for all $x\in \bar{I}$, but by the Taylor Theorem (see \citep[Theorem 5.6.2]{1971-Cartan}) and the hypothesis we have
\begin{equation*}
\|Df_{(u,v)}(x,u_1,v_1)-D f_{(u,v)}(x,u_0,v_0)\|_{\mathcal{L}(\R^2)}\leq
(c_1)^{-1}K \|(u_1,v_1)-(u_0,v_0) \|_{\R^2}
\end{equation*}
for all $x\in \bar{I}$, and $(u_1,v_1),\, (u_0,v_0)\in \bar{B}\left(0,c_1 r\right)_{\R^2}$. On the other hand, since $u,\, w\in\bar{B}\left(0,r\right)_{H^2_N(I)}$, it follows by item $(iii)$ of Lemma~\ref{lemmacont} that $(u(x),u'(x)), (w(x),w'(x))\in \bar{B}\left(0,c_1 r\right)_{\R^2}$ and therefore, letting $h=w-u\in H^2_N(I)$, by the last inequality and item $(iii)$ of Lemma~\ref{lemmacont} we have
\[
\|D f_{(u,v)}(x,u(x),u'(x))-D f_{(u,v)}(x,w(x),w'(x)) \|_{\mathcal{L}(\R^2)}\leq (c_1)^{-1}K \|(h(x),h'(x)) \|_{\R^2}
\]
which implies
\[
|((D\mathcal{F}(u)-D\mathcal{F}(w))(z))(x) |\leq (c_1)^{-1} K \|(h(x),h'(x)) \|_{\R^2} \|(z(x),z'(x)) \|_{\R^2}
\]
for all $x\in \bar{I}$ and since, again by item $(iii)$ of Lemma~\ref{lemmacont}, we have $\left\|h(x),h'(x)\right\|\leq c_1\left\|h\right\|_{H^2_N(I)}$ it follows that
\[
|((D\mathcal{F}(u) - D\mathcal{F}(w))(z))(x) | \leq K \left\|h \right\|_{H^2_N(I)}\left\|(z(x),z'(x))\right\|_{\R^2}
\]
for all $x\in \bar{I}$, and thus, taking the $L^2(I)$ norm in the above inequality we obtain
\[
\left\|(D\mathcal{F}(u) - D\mathcal{F}(w))(z) \right\|_{L^2(I)} \leq K\left\|h\right\|_{H^2_N(I)}\left\|z\right\|_{H^2_N(I)}.
\]
Finally, since $z\in H^2_N(I)$ was arbitrarily chosen we conclude from the last inequality that
\[
\|D\mathcal{F}(u)-D\mathcal{F}(w) \|\leq K  \|h \|_{H^2_N(I)} = K  \|u-w \|_{H^2_N(I)},
\]
proving the theorem.
\end{proof}

\section{Applications}
\label{sec:Applications}

As an application of Theorem~\ref{strong}, we present some examples where we rigorously verify the existence of a true solution for $\mathcal{F}(u) = 0$ near an approximate solution in a closed ball $\bar{B}(0,r)_{H^2(I)}$ for $\mathcal{F} \colon H^2_N(I) \to L^2$ given by $\mathcal{F}(u) = u'' - f(x,u,u')$. The rigorous verification of a solution $w\in \bar{B}(0,r)_{H^2_N(I)}$ is done by computing rigorously bounds for $\eta$, $\nu$ and $K$ in Theorem~\ref{kantorovich2}. The code to verify the solutions in the examples below are available at \citep{ramos-eduardo-2020-4310019}.

\subsection{Computation of a numerical solution $w$}

We compute a numerical solution $w$ of $\mathcal{F}(w)=0$ by numerically computing a zero $b$ of the nonlinear system of equations $\mathcal{F}_{\cos,m}(b)=0$, where $\mathcal{F}_{\cos,m} \colon \R^m\to \R^m$ is given in Definition~\ref{F_cosm} and denoted component-wise by $\mathcal{F}_{\cos,m}=(G_1,\cdots,G_n)$, where $G_i \colon \R^m\to \R$ is given by
\begin{equation*}
G_1(b)=\int_0^1 \left(w(b)''(x) - f(x,w(b)(x),w(b)'(x))\right)dx
\end{equation*}
and
\begin{equation*}
\begin{aligned}
G_i(b)=\int_0^1 \left( w(b)''(x) - f(x,w(b)(x),w(b)'(x))\right)\sqrt{2}\cos((i-1)\pi x) dx
\end{aligned}
\end{equation*}
for $2\leq i\leq m$, where $w(b)(x)=h^{-1}_{\cos,m}(b)(x) = b_1 + \sum_{i=2}^m \sqrt{2} b_i \frac{\cos((i-1)\pi x)}{\omega(i)}$.

\subsection{Computation of $\eta$}

Since the numerical solution $w$ is an elementary function, we can compute an interval enclosure $\hat{\eta}$ for
\begin{equation*}
\left\|\mathcal{F}(w)\right\|_{L^2} = \sqrt{\int_0^1 \left(w''(x) - f(x,w(x),w'(x))\right)^2 dx}
\end{equation*}
using the Simpson rule with explicit error bounds (see \citep[Theorem~12.1]{2010-Rump}) and thus we let
\[
\eta=\sup(\hat{\eta})\geq \left\|\mathcal{F}(w)\right\|_{L^2}.
\]

\subsection{Computation of $\nu$}

The computation of a bound for $\lambda(D\mathcal{F}(w))$ is done as follows. First we prove that $D\mathcal{F}(w)_{\cos,m}$ is invertible. Then it follows by definition that
\begin{equation*}
\lambda(D\mathcal{F}(w)_{\cos,m})=
\left\|D\mathcal{F}(w)_{\cos,m}^{-1}\right\|_2^{-1}\geq \left\|D\mathcal{F}(w)^{-1}_{\cos,m}\right\|_F^{-1}
\end{equation*}
where $\left\|.\right\|_F$ is to the Frobenius norm. Then we can use the bounds for the Simpson rule (see \citep{2010-Rump}) to rigorously compute integrals enclosing the coordinates of $D\mathcal{F}(w)_{\cos,m} \colon \R^m\to \R^m$ given by Corollary~\ref{cor3} and thus we compute an interval enclosure $\hat{\nu}_0$ for
\begin{equation*}
\left\|D\mathcal{F}(w)^{-1}_{\cos,m}\right\|_F^{-1}
\end{equation*}
and let $\nu_0 = \inf(\hat{\nu}_0)\leq \lambda(D\mathcal{F}(w)_{\cos,m})$. To compute $N$ satisfying \eqref{eq:estimateN} in Theorem~\ref{strong}, we directly compute an interval enclosure $\hat{N}$ for
\[
\left\|f_u(z)\right\|_{C^0(I)}+\left\|f_u(z)\right\|_{C^1(I)}+\left\|f_{u'}(z)\right\|_{C^0(I)}+\left\|f_{u'}(z)\right\|_{C^1(I)},
\]
where $z = z(x)=(x,w(x),w'(x))$, and let $N = \sup{(\hat{N})}$. Finally, we compute an interval enclosure $\hat{\nu}$ for
\begin{equation*}
\min \left\{ \nu_0, \frac{(\pi m)^2}{\omega(m+1)} \right \} - \frac{N}{\pi m}
\end{equation*}
and let $\nu=\inf(\hat{\nu})$. From Theorem~\ref{strong} it follows that $\nu \leq \lambda(D\mathcal{F}(w))$, thus $\nu$ is the desired lower bound.

\subsection{Computation of $r$ and $K$}

In our examples we use $r=\left\|w\right\|_2$, and 
$K$ is computed directly from item $(ii)$ of Theorem~\ref{strong} as follows. We first estimate
\[
\left\|D_{(u,v)}^2f(x,u,v)\right\| \leq \sqrt{(f_{uu}(x,u,v))^2 + 2 (f_{uv}(x,u,v))^2 + (f_{vv}(x,u,v))^2}
\]
and then using interval arithmetic we compute an upper bound $K$ for
\[
c_1 \sqrt{(f_{uu}(x,u,v))^2 + 2 (f_{uv}(x,u,v))^2 + (f_{vv}(x,u,v))^2}
\]
in $\bar{I} \times \bar{B}\left(0,c_1 r\right)_{\R^{2}}$, where $c_1 = \sqrt{\frac{e^2+1}{e^2-1}}$. It then follows that $c_1 \left\|D_{(u,v)}^2f(x,u,v)\right\| \leq K$ as desired.

\begin{example}
\label{example1}\normalfont
In this example we shall consider the following
equation
\begin{equation*}
u''- \lambda(g(u) + c(x)) = 0,\ u'(0) = u'(1) = 0
\end{equation*}
where $g \colon \R\to \R$ and $c \colon \R\to \R$ are $C^2$ functions.

In this case the hypothesis on Theorem~\ref{strong} for $N$ and $K$ translate to $\left\|g_u(w)\right\|_{C^0(I)}+\left\|g_u(w)\right\|_{C^1(I)}\leq N$ and $c_1\left|g_{uu}(u)\right|\leq K$ for all $u\in [-c_1r,c_1r]$. We consider two specific cases of this equation as follows
\begin{enumerate}[i)]
\item Given $\mathcal{F}_1:H^2_N(I)\to L^2(I)$ defined by $\mathcal{F}_1(u)=u'' - \left(-u+\frac{u^3}{6}-\cos(\pi x)\right)$ and letting $R=1$, we found $u_1$ with $m=5$ Fourier coefficients given by
\begin{equation*}
u_1(x) = 0.82942 \frac{\sqrt{2}\cos(\pi x)}{\omega(2)} - 4.2928 \times 10^{-5}\frac{\sqrt{2}\cos(3\pi x)}{\omega(4)}
\end{equation*}
as a numerical approximation for a zero of $\mathcal{F}_1$ 
(see Figure~\ref{fig:solns}(a)). We then computed, as described above, the bound
\begin{equation*}
K := 1.0436
\end{equation*}
for Theorem~\ref{strong}, and we computed
\begin{equation*}
\|\mathcal{F}_1(u_1) \|_{L^2(I)} \leq 5.1907 \times 10^{-8} =: \eta
\end{equation*}
and
\begin{equation*}
\lambda( D\mathcal{F}_1(u_1) ) \geq \lambda(D\mathcal{F}_1(u_1)_{\cos,m})-\frac{N}{5\pi}\geq 0.29845 =: \nu.
\end{equation*}
Using these bounds we conclude by Theorem~\ref{kantorovich2} that there exists a zero $u^*_1$ of $\mathcal{F}_1$ such that
\begin{equation*}
\left\|u_1-u^*_1\right\|_{H^2_N(I)}\leq \frac{\nu-\sqrt{\nu^2 - 2K\eta}}{K} \leq 1.7392 \times 10^{-7}.
\end{equation*}
The computation and the verification of this solution took $25.53$ seconds on a $2.3$ GHz Quad-Core processor. \\

\item For $\mathcal{F}_2 \colon H^2_N(I)\to L^2(I)$ defined by $\mathcal{F}_2(u)=u'' - \left(\sin(u) - \cos(2\pi x)\right)$, once again for $m=5$ and $R=1$, we computed a numerical solution $u_2$ given by
\begin{equation*}
u_2 (x)= 3.1416 + 0.73507 \frac{\sqrt{2}\cos(2\pi x)}{\omega(3)} + 2.1943 \times 10^{-11} \frac{\sqrt{2}\cos(4\pi x)}{\omega(5)}
\end{equation*}
as an approximation for a zero of $\mathcal{F}_2$ 
(see Figure~\ref{fig:solns}(b)). We then computed the bound
\begin{equation*}
K := 1
\end{equation*}
for Theorem~\ref{strong}, and the bounds
\begin{equation*}
\|\mathcal{F}_2(u_2) \|_{L^2(I)}\leq  5.7081 \times 10^{-7} =: \eta
\end{equation*}
and
\begin{equation*}
\lambda(D\mathcal{F}_2(u_2))\geq \lambda(D\mathcal{F}_2(u_2)_{\cos,m}) - \frac{N}{m\pi}\geq 0.29867 =: \nu.
\end{equation*}

We then concluded from Theorem~\ref{kantorovich2} that there exists a zero $u_2^*$ of $\mathcal{F}_2$ such that
\begin{equation*}
\|u_2 - u^* \|_{H^2_N(I)} \leq \frac{\nu-\sqrt{\nu^2 - 2K\eta}}{K} \leq 1.9112 \times 10^{-6}.
\end{equation*}
The computation and the verification of this solution took $26.27$ seconds on a $2.3$ GHz Quad-Core processor.

\end{enumerate}
\end{example}

\begin{example}
\label{CHERPION2001}\normalfont
Consider $\mathcal{F}_3 \colon H^2_N(I)\to L^2(I)$ defined by $\mathcal{F}_3(u)(t)=u''(t) -\left( u(t) - (u'(t))^2 + \sin(t) \right)$, for $u\in H^2_N(I)$ and $t\in I$. The two-point boundary value $\mathcal{F}_3(u)(t)=0$ was studied using iteration methods in \citep{CHERPION200175} where a lower solution $\alpha(t)= -1$ and an upper solution $\beta(t)=1$ were computed satisfying $\alpha(t) \leq u^*(t) \leq \beta(t)$ for all $t \in I$, where $u^*$ is the true solution.

We applied our method with $m=8$ and $R=0.5$ using the numerical solution
\begin{equation*}
\begin{aligned}
u_3 (x) = & -0.4546385
+ 0.2353235 \frac{\sqrt{2}\cos(\pi x)}{\omega(2)} + \\
& 0.0134243 \frac{\sqrt{2}\cos(2\pi x)}{\omega(3)}
+ 0.0244880 \frac{\sqrt{2}\cos(3\pi x)}{\omega(4)} + \\
& 0.0039254 \frac{\sqrt{2}\cos(4\pi x)}{\omega(5)}
+ 0.0088177 \frac{\sqrt{2}\cos(5\pi x)}{\omega(6)} + \\
& 0.0017909 \frac{\sqrt{2}\cos(6\pi x)}{\omega(7)}
+ 0.0044966 \frac{\sqrt{2}\cos(7\pi x)}{\omega(8)}
\end{aligned}
\end{equation*}
as an approximation for a zero of $\mathcal{F}_3$ 
(see Figure~\ref{fig:solns}(c)). For this approximate solution we compute, as above, the bounds $K = 2$, $\eta = 0.0041703$, and $\nu = 0.18818$. Applying Theorem~\ref{kantorovich2} we get that there exists a true zero $u_3^*$ of $\mathcal{F}_3$ satisfying
\begin{equation*}
\|u_3 - u^*_3 \|_{H^2_N(I)} \leq t^* = 0.02566,
\end{equation*}
and $u^*_3$ is the unique zero in $\bar{B}(u_3,t^{**})_{H^2_N}$, where $t^{**} = 0.162522$. Hence we get local uniqueness and a good estimate of the true solution of the boundary value problem. The computation time for this example was $103.56$ seconds on a $2.3$ GHz Quad-Core processor.
\end{example}

\begin{example}\label{Reviewer}
\normalfont
Consider $\mathcal{F}_4 \colon H^2_N(I)\to L^2(I)$ defined by $\mathcal{F}_4(u)(t)=u''(t) + u(t)^2 + u(t) - |t-1/2|$, for $u\in H^2_N(I)$ and $t\in I$. We apply our method to the boundary value problem $\mathcal{F}_4(u)(t)=0$ with $m=8$ and $R=0.5$ and using the numerical solution
\[
u_4 (x) = 0.2070967
-0.1505309 \frac{\sqrt{2}\cos(2\pi x)}{\omega(3)}
-0.0160072 \frac{\sqrt{2}\cos(6\pi x)}{\omega(7)}
\]
as an approximate zero of $\mathcal{F}_4$ 
(see Figure~\ref{fig:solns}(d)). For this approximate solution we compute the bounds $K = 4$, $\eta = 0.0069201$, and $\nu = 0.23533$. Applying Theorem~\ref{kantorovich2} we get that there exists a true zero $u_4^*$ of $\mathcal{F}_4$ satisfying
\begin{equation*}
\|u_4 - u^*_4 \|_{H^2_N(I)} \leq t^* = 0.057703,
\end{equation*}
and $u^*_4$ is the unique zero in $\bar{B}(u_4,t^{**})_{H^2_N}$, where $t^{**} = 0.059963$. The computation time for this example was $110.21$ seconds on a $2.3$ GHz Quad-Core processor.
\end{example}

\begin{figure}
\centering
\begin{subfigure}{0.35\textwidth}
\centering
\includegraphics[width=\textwidth]{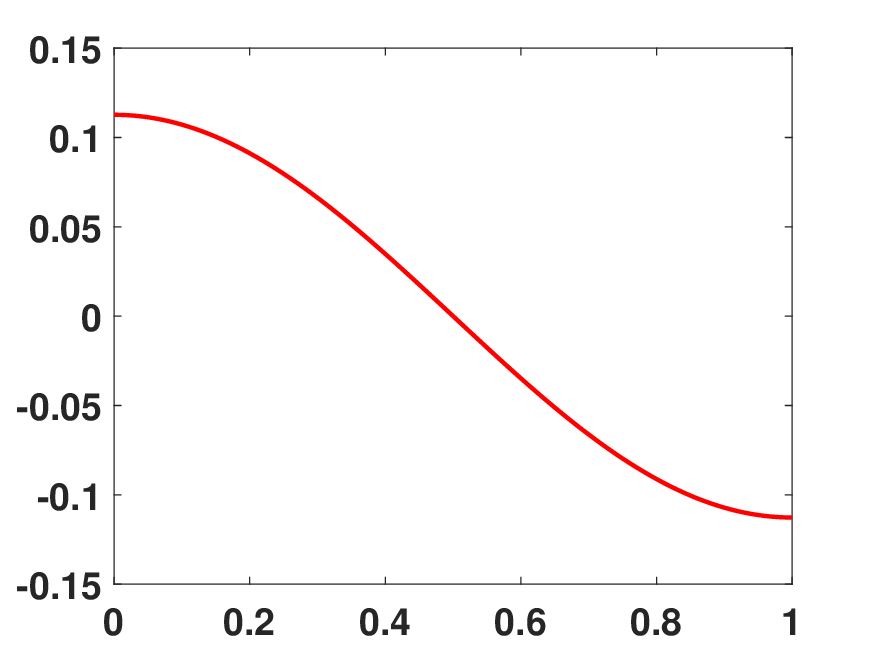}
\caption{Solution $u_1$}
\end{subfigure}
\begin{subfigure}{0.35\textwidth}
\centering
\includegraphics[width=\textwidth]{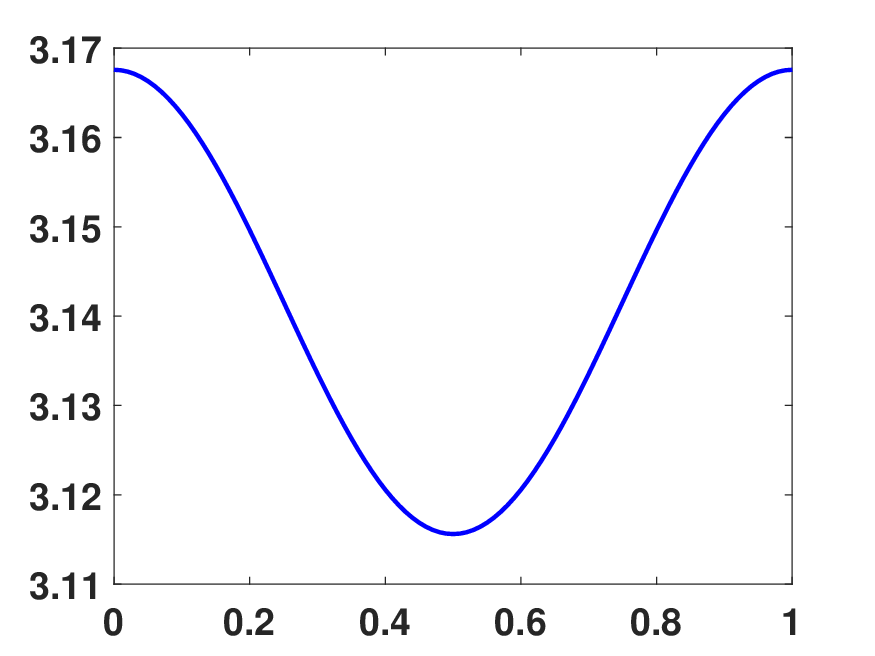}
\caption{Solution $u_2$}
\end{subfigure}
\begin{subfigure}{0.35\textwidth}
\centering
\includegraphics[width=\textwidth]{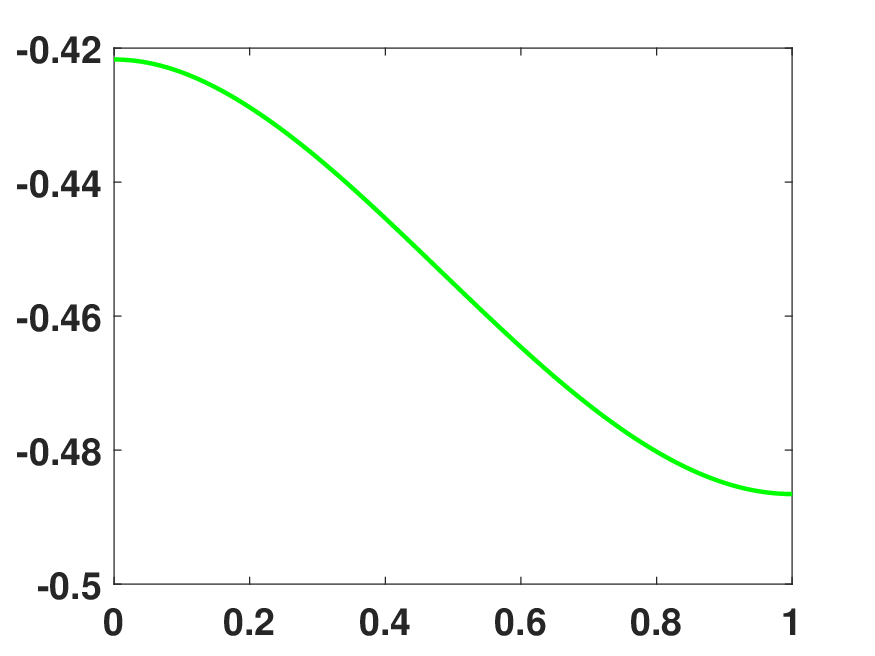}
\caption{Solution $u_3$}
\end{subfigure}
\begin{subfigure}{0.35\textwidth}
\centering
\includegraphics[width=\textwidth]{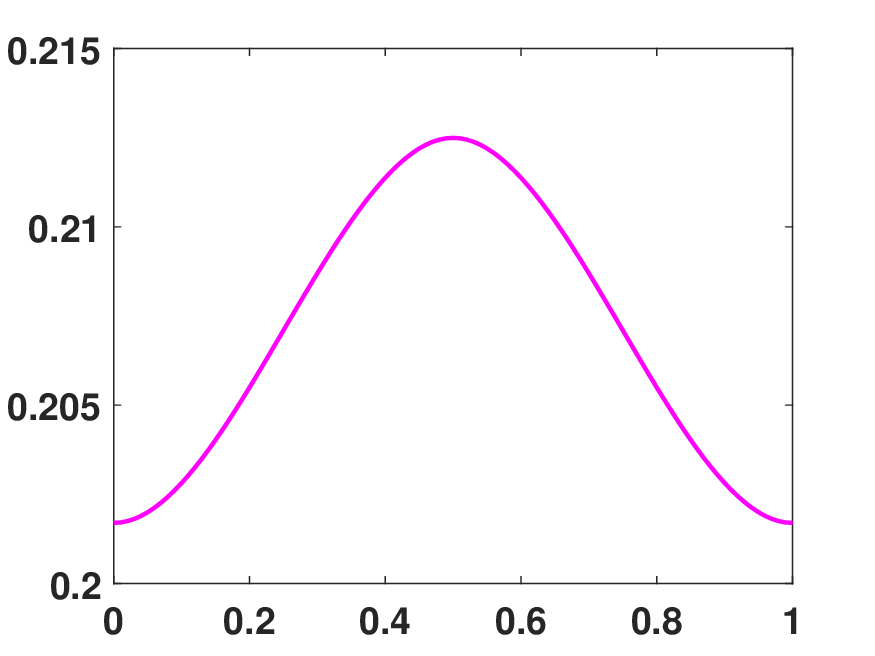}
\caption{Solution $u_4$}
\end{subfigure}
\caption{Numerical zeros $u_i$ of $\mathcal{F}_i$ in Examples~\ref{example1}, \ref{CHERPION2001} and \ref{Reviewer}}
\label{fig:solns}
\end{figure}

\begin{remark}
Notice that we can use Lemma~\ref{lemmacont} to get bounds on the norm
$\left\| \cdot \right\|_{C^1(\bar{I})}$ in terms of the norm $\left\| \cdot \right\|_{H^2_N(I)}$ of a solution. Hence we can characterize the precision of our approximate solutions with respect to the $C^1(I)$ norm.
\end{remark}

\section{Conclusion}
\label{sec:Conclusion}

As shown in Examples~\ref{example1}, \ref{CHERPION2001}, and \ref{Reviewer}, our method was able to verify the existence of true solutions near numerical solutions for two-point boundary value problems with small error bounds (less than $10^{-5}$). In Example~\ref{CHERPION2001} our method was able to get better localization of the true solution than the iterative method in \citep{CHERPION200175}. The method is easy to implement and the code is available at \citep{ramos-eduardo-2020-4310019}. Theorem~\ref{strong} together with Theorem~\ref{kantorovich2} are easy to apply because the bounds can be calculated via numerical integration with interval arithmetic. This motivated our choice to use the bijectivity modulus to verify the solutions.

\section*{Acknowledgments}
The work of M.G. was partially supported by FAPESP grant 2019/06249-7, by CNPq grant 309073/2019-7, by the National Science Foundation under awards DMS-1839294 and HDR TRIPODS award CCF-1934924, DARPA contract HR0011-16-2-0033, and National Institutes of Health award R01 GM126555.

\bibliographystyle{elsarticle-num} 

\end{document}